\documentclass{article}

\usepackage[noblocks]{authblk}

\usepackage{makeidx}     
\usepackage{graphicx,epsfig,lscape,float}    
\usepackage{amsmath,array,amssymb, amsfonts, latexsym,amsthm, mathrsfs}
\usepackage{psfrag,fancybox} 
\usepackage{array}
\usepackage{url}
\usepackage{tikz,pgfplots,circuitikz,pgfplotstable} 
\usepackage{tkz-euclide}
\usetkzobj{all}

\usepackage{multicol}   
\usepackage{enumerate}   
\usepackage{accents}
\usepackage{listings}

\makeindex             
\DeclareMathOperator{\sinc}{\mathbin{\operatorname{sinc}}}
\DeclareMathOperator{\sinhc}{\mathbin{\operatorname{sinhc}}}
\DeclareMathOperator{\artanh}{\mathbin{\operatorname{artanh}}}
\DeclareMathOperator{\artanhc}{\mathbin{\operatorname{artanhc}}}

\newcommand{\rabs}{\mathbf{abs}}
\newcommand{\R}{\mathbb{R}} 

\newcommand{\cD}{{\cal D}}

\newcommand{\hx}{\hat x}
\newcommand{\hv}{\hat v}
\newcommand{\hz}{\hat z}
\newcommand{\hu}{\hat u}
\newcommand{\hy}{\hat y}
\newcommand{\hw}{\hat w}
\newcommand{\hF}{\hat F}
\newcommand{\cx}{\check x}
\newcommand{\cv}{\check v}
\newcommand{\cz}{\check z}
\newcommand{\cu}{\check u}
\newcommand{\cy}{\check y}
\newcommand{\cw}{\check w}
\newcommand{\cF}{\check F}
\newcommand{\rx}{\mathring x}
\newcommand{\rv}{\mathring v}
\newcommand{\rz}{\mathring z}
\newcommand{\ru}{\mathring u}
\newcommand{\ry}{\mathring y}
\newcommand{\rw}{\mathring w}
\newcommand{\rc}{\mathring c}
\newcommand{\rF}{\mathring F}

\newcommand{\mathrad}{\delta\!}

\renewcommand{\epsilon}{\varepsilon}

\theoremstyle{plain}
\newtheorem{theorem}{Theorem}[section]
\newtheorem{corollary}[theorem]{Corollary}
\newtheorem{lemma}[theorem]{Lemma}
\newtheorem{proposition}[theorem]{Proposition}

\theoremstyle{definition}
\newtheorem{definition}{Definition}

\theoremstyle{remark}
\newtheorem{remark}{Remark}

\begin{document}
\title{Piecewise Linear Secant Approximation via \\
Algorithmic Piecewise Differentiation}

 \author[1]{Andreas Griewank}
 \author[2,3]{Tom Streubel}
 \author[3]{Lutz Lehmann}
 \author[4]{Manuel Radons}
 \author[3]{Richard Hasenfelder}
 
 \affil[1]{School of Mathematical Sciences and Information Technology, Ecuador}
 \affil[2]{Zuse Institute Berlin, Germany}
 \affil[3]{Humboldt University of Berlin, Germany}
 \affil[4]{Technical University of Berlin, Germany}
 


\maketitle

\noindent 
\begin{abstract}
It is shown how piecewise differentiable functions $F:  \R^n \mapsto \R^m $ that are defined by
evaluation programs can be approximated locally by a piecewise linear
model based on a pair of sample points $\check x$ and $\hat x$.  
We show that the discrepancy between function and model at any point $x$ is of
the bilinear order $\mathcal O(\|x-\cx\| \|x-\hx\|)$.
As an application of the piecewise linearization procedure we devise a generalized 
Newton's method based on successive piecewise linearization and 
prove for it sufficient conditions for convergence and convergence rates equaling those of 
semismooth Newton. We conclude with the derivation of formulas for the numerically stable implementation 
of the aforedeveloped piecewise linearization methods.
\end{abstract}

\vspace*{5pt}
 \vspace{4pt}


\section{Introduction}

In this paper we refine and extend the theory of piecewise linearizations\footnote{Our notion of linearity includes
nonhomogeneous functions, where the adjective {\em affine} or perhaps {\em
polyhedral} would be more precise. However, such mathematical terminology might
be less appealing to computational practicioners and to the best of our
knowledge  there are no good nouns corresponding to {\em linearity} and {\em
linearization} for the adjectives affine and polyhedral.} of piecewise differentiable functions $F:\R^n\mapsto\R^m$ that was introduced in \cite{mother}. 
Throughout, we assume that any such function is defined by an evaluation procedure consisting of a sequence of
elemental functions $ v_i = \varphi_i(v_j)_{j\prec i}$, where $\varphi_i$ is either smooth or the absolute value function. The data dependence relation $\prec$ generates a partial ordering, 
which yields a directed acyclic graph.  
In other words, we assume that we have a straight line program without any
loops or jumps in the control flow.

In the above reference a piecewise linearization $\Delta
F(\rx;\Delta x)$ of $F$ about the point $\rx$ is constructed such that, for arbitrary $x \in
\R^n$  with $\rF = F(\rx)$, it holds
\[ F(x) \; = \;  \rF + \Delta F(\rx;x-\rx ) +  \mathcal O( \|x -\rx \|^2)\; . \]
Since  $\Delta F(\rx;\Delta x)$ is constructed by
replacing smooth elementals by their tangent lines at $\rx$, we will refer to it
as \textbf{piecewise tangent linearization}. An important property of these piecewise linearizations is that they
vary continuously with respect to the sample point or points at which they are developed.

As a generalization of $\Delta F(\rx;\Delta x)$ we construct a
\textbf{piecewise secant linearization} $\Delta F(\cx, \hx;\Delta x)$ of $F$ at
the pair $\cx$ and $\hx$  such that for the midpoints $\rx = (\cx+\hx)/2$ and
$\rF = (\cF + \hF)/2$ with  $\cF = F(\cx)$ and $\hF = F(\hx)$ it holds 
\[ F(x )\; = \; \rF + \Delta F(\cx, \hx;x-\rx ) \, + \, {\cal O}( \| x- \cx \|
\| x- \hx \| )\;.\]
The new $\Delta F(\cx, \hx;
\Delta x)$ reduces to $\Delta F(\rx;\Delta x)$ when the two sample points $\cx$ and $\hx$ coalesce at $\rx$.

Our results in this article are twofold: 

\begin{itemize}
 \item Firstly, we present approximation properties and Lipschitz continuity estimates. Here the two major points are that 
the piecewise linearization is a second order approximation to the underlying function. Moreover, we do not only prove the existence of Lipschitz constants, but provide explicit estimates.
These results immediately yield additional statement on conditions for perturbation-stable surjectivity of a piecewise linear model.
 \item Secondly, we develop an application for the piecewise linearization. Namely, we introduce, for each linearization mode, a Newton's method based on successive piecewise linearization and give 
sufficient conditions for convergence, as well as statements on the convergence rates.
\end{itemize}
In an appended section we moreover provide formulas for singularity free and thus numerically stable 
implementations of the piecewise secant linearizations.

One significant advantage of the piecewise linearization-approach is that, once the piecewise linearization is generated, we also know where its kinks are located. 
This  means that we are liberated of the complications of event handling that the nondifferentiabilities of piecewise smooth functions usually bring about, 
which was one of the motivations for the development of the techniques presented in \cite{mother}. In said reference it is proved that the variations 
of the piecewise linear model are Lipschitz continuous with respect to perturbations of the development point. 
The results mentioned in the first bullet point represent a significant improvement over those given in \cite{mother} in so far as they are not only sharper, 
but also provide explicit bounds for the Lipschitz constants for the variations of the model with respect to perturbations of the base point.


\subsection*{Content and Structure}

To provide some more background, we will, in the second section, elaborate on the setting outlined above both from the mathematical and the implementation perspective. 
In Section 2 we will introduce the piecewise linearization framework that is investigated thereafter. In Section 3 we derive the approximation and stability properties of said framework. In the subsequent section the generalized Newton's methods are derived for both linearization modes. 
In Section 5 we show that the secant linearization can be
computed in a division free, centered form such that it reduces continuously to
the tangent mode when the reference points $\cx$ and $\hx$ coalesce. The latter is followed by a numerical example in Section 6. The article is concluded by some final remarks.


\subsection*{Preliminary Remarks}

Let $\tilde\Phi$ be a library of elemental functions that conforms to the condition of elemental differentiability as described in \cite{buch}.  
In the following we will consider piecewise differentiable functions that are defined by an evaluation procedure consisting of a sequence of
such elemental functions $ v_i = \varphi_i(v_j)_{j\prec i}$, where the $\varphi_i$ are contained in a library $$\Phi:=\tilde\Phi\; \cup\; \{\rabs()\}\; .$$ 
We remark that the inclusion of the absolute value function into the library means that we can also evaluate $\min()$ and $\max()$ via the identities 
\begin{equation}
 \max(u,w) = (u+w+\rabs(u-w))/2 \; , \quad \min(u,w) = (u+w-\rabs(u-w))/2\; .
\end{equation}
There is a slight implicit restriction, namely, we assume that whenever $\min$
or $\max$ are evaluated both their arguments have well defined finite values so
that the same is true for their sum and difference.
On the other hand, the expression $\min(1,1/\rabs(u))$ makes perfect 
sense in IEEE arithmetic \cite{ieee}, but rewriting it as above leads to a $NaN$ at $u=0$. 
While this restriction may appear quite technical, it imposes the requirement
that all relevant quantities are well defined at least in some open
neighborhood, which is exactly in the nature of piecewise differentiability. 
For an in-depth investigation of piecewise differentiable functions, see, e.g., the books by Kummer \cite{kummer} and Scholtes \cite{scholtespl}. 


The function $\Delta F(\rx;\Delta x)$ is incremental in that $\Delta F(\mathring x;0) =0$,
but like general piecewise linear continuous functions, it is only locally and
positively homogeneous so that 
\[ \Delta F(\rx;\alpha \, \Delta x) \; = \; \alpha \, \Delta F(\rx;\Delta x)
\quad \mbox{for} \quad 0 < \alpha \|\Delta x\| < \rho(\rx)\;. \] 
Here the bound $\rho(\rx)$ is positive everywhere, but generally not
continuous with respect to $\rx$.

In \cite{mother} it has been shown that the Jacobians of the linear pieces of 
$\Delta F(\rx;\alpha \Delta x)$ in the ball of radius $\rho(\rx)$ about $\rx$
are conically active generalized Jacobians of the underlying nonlinear function
$F(x)$. We will not elaborate on this connection here,
because even in the smooth case secant approximations need not correspond to
exact Jacobians.
In contrast, the generalized derivative sets in the sense of Clarke \cite{clarke}  are for
piecewise differentiable functions almost everywhere just singletons, containing
the classical Jacobian matrix. In floating point arithmetic the user or client
algorithm will then quite likely never {\em 'see'} the nonsmoothness or gain any
useful information about it.
 
\section{Piecewise Linearization by Tangents and Secants } 
Suppose the vector function $F : \cD \subset \R^n \rightarrow \R^m$ in question is evaluated by a 
sequence of assignments 
$$ v_i = v_j \circ v_k \quad \mbox{or} \quad   v_i =  \varphi_i(v_j) \quad 
\mbox{for} \quad i=1 \ldots l \; . $$
Here $\circ \in \{ +,-,*,/ \}$ is a polynomial arithmetic operation and 
$$ \varphi_i \in \Phi \; \equiv \; \{\mbox{rec, sqrt},  \sin, \cos, \exp, \log, \ldots , \rabs, \ldots \}   $$
a univariate function. The user or reader may extend the library by other locally Lipschitz-continuously 
differentiable functions like the analysis favorites 
$$ \varphi(u) \equiv |u| > 0 \; ?\;  u^p\cdot  \sin(1/u) : 0 \quad  \mbox{for}  \quad p \geq 3\; . $$ 
To fit into the the framework they then also have to supply an evaluation procedure for both 
the elemental function $\varphi$ and its derivative  $\varphi^\prime$, 
which cannot be based mechanically on the chain rule. 

Following the notation from \cite{buch}  we partition the sequence of scalar variables $v_i$ into the 
vector triple 
$$ (x,z,y) \; = \; \left ( v_{1-n}, \ldots , v_{-1}, v_{0}, \ldots , v_{l-m}, 
 v_{l-m+1},\ldots , v_l \right )  \, \in \, \R^{n+l}$$
such that $x \in \R^n $ is the vector of independent variables, $y \in \R^m $ the vector
of dependent variables and $z \in \R^{l-m}$ the (internal) vector of intermediates.  

Some of the elemental functions like the reciprocal $rec(u) \equiv 1/u$, the square root and
the logarithm are not globally defined. As mentioned above, we will assume that
 the input variables $x$ are 
restricted to an open domain ${\cal D} \subset \R^n$ such that all resulting intermediate values  $v_i =v_i(x)$ are 
well defined. 

 Throughout we will assume 
 that the evaluation procedure for $F$ involves exactly $s \geq 0$ calls to $\rabs()$, including $\min$ and 
 $\max$ rewritten or at least reinterpreted as discussed above.   Starting from $\rx$ and an increment 
 $\Delta x =  x -\rx$, we will now construct for each intermediate $v_i$ an  approximation 
 $$ v_i(\rx+\Delta x) -  \rv_i \; \approx \;  \Delta v_i  \equiv  \Delta v_i(\Delta x)\,. $$
 Here the incremental function $\Delta v_i(\Delta x)$ is continuous and piecewise linear, with $\rx$ or $\cx$ and $\hx$
 considered constant in the tangent and secant case, respectively.
 Hence, we will often list  $\Delta x$, but only rarely
 use $\rx$, $\cx$ and $\hx$ as arguments of the $\Delta v_i$ in proofs.  

\subsection*{Defining Relations for Tangent Approximation}
We use the reference values $\rv_i=v_i(\rx)$ and,  assuming that all $\varphi_i$ other than the absolute value
function are differentiable within the domain of interest, we may use the tangent linearizations 
\begin{eqnarray}
\Delta v_i  =  &\Delta  v_j \pm \Delta v_k &\quad \mbox{for} \quad v_i  =  v_j \pm v_k \,, \label{:linadd}\\
\Delta v_i  = & \rv_j \ast \Delta v_k +  \Delta v_j \ast \rv_k & \quad \mbox{for} \quad  v_i  =  v_j \ast v_k \,, \label{:linmult} \\
 \Delta v_i   =  & \rc_{ij}\ast \Delta v_j  
& \quad \mbox{for} \quad v_i= \varphi_i(v_j) \not \equiv \rabs() \,. \label{:linphi} 
\end{eqnarray}
Here $\rc_{ij} \equiv \varphi_i^\prime (\rv_j)$, which will be different for the secant linearization.    
 
If no absolute value or other 
nonsmooth elemental occurs, the function $y=F(x)$ is, at the current point, 
differentiable and by the chain rule we have the relation
$$ \Delta y \; = \Delta F(\rx;\Delta x) \; \equiv \;  F^\prime(\rx) \Delta x, $$
where $ F^\prime(x) \in \R^{m \times n}$ is the Jacobian matrix.
Thus we observe the obvious fact that smooth differentiation is 
equivalent to linearizing all elemental functions.

Now, let us move to the piecewise differentiable scenario, where  the absolute value function 
does occur $s>0$ times. We may then obtain a piecewise
linearization of the vector function $F(\rx+\Delta x)- F(\rx) $ by incrementing 
\begin{equation}
 \Delta v_i = \rabs(\rv_j + \Delta v_j) - \rv_i \quad \mbox{when} \quad v_i =
\rabs(v_j) \; .     \label{:absdiff} 
\end{equation}
Here $\rv_i=  \rabs(\rv_j)$, which will be slightly different for the secant linearization. 
In other words, we keep the piecewise linear function $\rabs()$ unchanged so that the 
resulting $\Delta y$ represents, for each fixed $x \in \cD$, the piecewise linear and continuous {\em increment function}
$$  \Delta y \; = \;  \Delta y(\Delta x) \; = \;  \Delta F(\rx;\Delta x) \; : \; \R^n \rightarrow \R^m \; . $$

\subsection*{Defining Relations for Secant Approximation}
In the tangent approximation the reference point was always the evaluation point $\rx$ and the resulting 
values $\rv_i = v_i(\rx)$. Now we will make reference to the midpoints
\begin{equation}
\rv_i \; \equiv \;  (\cv_i + \hv_i)/2 \quad \mbox{of} \quad \cv_i \equiv v_i(\cx)     \quad \mbox{and} \quad \hv_i \equiv v_i(\hx) \,.
\end{equation}
Consequently, we have the functional dependence $\rv_i \, = \, \rv_i(\cx, \hx)$, which is at least Lipschitz continuous under our assumptions.
Now an intriguing observation is that the recurrences \eqref{:linadd} and \eqref{:linmult} for arithmetic operations can stay just the 
same, and the recurrence \eqref{:linphi} for nonlinear univariates is still formally valid, except that the tangent slope 
$\varphi^\prime(\rv_j)$ must be replaced by the secant slope 
\begin{equation}
  c_{ij} \equiv  \left \{ \begin{array}{ll} (\cv_i-\hv_i)/(\cv_j-\hv_j)  & \; \mbox{if} \; \cv_j \ne \hv_j \\ \varphi_i^\prime (\rv_j)
             &  \; \mbox{otherwise} \end{array} \,. \right.   \label{:cdef} 
\end{equation} 
Theoretically, some $\cv_i$ and  $\hv_i$ may coincide, even if the underlying sample points $\cx$ and  $\hx$ are not selected 
identically, in which case the secant based model would reduce to the tangent based model.  
While exact coincidence of any pair $\cv_i$ and  $\hv_i$ is rather unlikely, taking the divided 
difference over small increments is likely to generate numerical cancellation. Therefore we will develop a division free centered form in Section \ref{appendix}. 
Finally,  the nonsmooth rule \eqref{:absdiff} can stay unchanged except that we now set
\begin{equation}
  \rv_i \;  \equiv  \;  \tfrac{1}{2} (\cv_i+\hv_i) \; = \; \tfrac{1}{2} [\rabs(\cv_j)+ \rabs(\hv_j)]  \,.  \label{:midef} 
\end{equation}
Hence, it is immediately clear that the 
new secant approximation reduces to the old tangent approximation when $\cx = \hx$. In general, we will denote the 
mapping between the input increments $\Delta x \in \R^n$ and the resulting values $\Delta y \in \R^m$ by 
$$  \Delta y \; = \;  \Delta y(\Delta x) \; = \;  \Delta F(\cx, \hx ;\Delta x) \; : \; \R^n \rightarrow \R^m \; . $$ 
Its piecewise linear structure is very much the same as that of the tangent based model, which is described in detail in 
\cite{mother}.  Here we emphasize its quality in approximating the underlying
nonlinear and nonsmooth $F$.
  

In contrast to the tangent model, the secant model is not a priori unique in that it depends quite strongly on the procedural representation of the vector function $F$ and not just its values, i.e. its properties as a mathematical map. For example, one can easily check that applying the above secant modeling rules to $f(x) = \log(\exp(x))$ does not yield the same approximation $\Delta f(\cx,\hx,\Delta x)$ as the one for $f(x)=x$.
On the other hand the natural secant linearization rule for the product $v = u \cdot w$ is equivalent to that obtained by applying the Appolonius identity
$$    u \cdot w  \; = \; \tfrac{1}{4} \left [ (u+w)^2   -  (u-w)^2 \right ] . $$
Of course, the same is true for the tangent linearization and we may assume without loss of generality that we only have 
three kinds of elemental functions, the addition, the modulus and smooth univariate functions. That reduction greatly simplifies the thoretical 
analysis but might not be numerically optimal for actual implementations.  
%
%

\section{Approximation and Stability Properties}\label{secApprox}
In contrast to the presentation in our previous papers we will now also use the nonincremental forms
 $$ \lozenge_{\rx}  F(x)     \; \equiv \;  F(\rx) + \Delta F(\rx; x-\rx )  $$
 and  
  $$ \lozenge_{\cx}^{\hx}  F(x)     \; \equiv \;  \tfrac{1}{2} ( F(\cx) + F(\hx))  + \Delta F(\cx, \hx; x-\rx ) \,.    $$
For the square as a univariate nonlinear function \(v(x) = x^2\) we find:
\begin{align}
	\lozenge_{\cx}^{\hx} v \;=\; \lozenge_{\cx}^{\hx} x^2 &= \tfrac 12[\cx^2 + \hx^2] + \frac{\hx^2 - \cx^2}{\hx - \cx}(x - \tfrac 12[\hx + \cx]) 
	= \tfrac 12 [\cx^2 + \hx^2] + (\hx + \cx)(x - \rx) \notag \\
	&= \tfrac 12 [\cx^2 + \hx^2] + 2\rx(x - \rx) = \mathring v + 2\rx(x-\rx). \label{eqn:squareFormula}
\end{align}         

\begin{lemma}
Plugging the secant approximation for the square into the Appolonius identity, we obtain for the general multiplication \(v = u\cdot w\):
\begin{align*}
 \lozenge_{(\cu,\cw)}^{(\hu,\hw)} (u\cdot w)    & 
 \; = \;    \tfrac{1}{4} \left [  \lozenge_{(\cu,\cw)}^{(\hu,\hw)}  (u+w)^2   -   \lozenge_{(\cu,\cw)}^{(\hu,\hw)}  (u-w)^2 \right ]  \\        
  & \; = \; \tfrac 12(\hu\hw + \cu\cw) + \rw  (  u - \ru ) +  \ru  (  w   - \rw ).
 \end{align*} 
\end{lemma}
\begin{proof}
\begin{align*}
	4\cdot \lozenge_{(\cu,\cw)}^{(\hu,\hw)} (u\cdot w)    & 
	\; = \;    \left [  \lozenge_{(\cu,\cw)}^{(\hu,\hw)}  (u+w)^2   -   \lozenge_{(\cu,\cw)}^{(\hu,\hw)}  (u-w)^2 \right ]  \\         
	&\; =\;\quad \tfrac 12\left([\hu + \hw]^2 + [\cu + \cw]^2\right) + 2(\ru + \rw)\left( u+w\, -\, [\ru + \rw] \right) \\
	&\qquad\; -\; \tfrac 12\left([\hu - \hw]^2 + [\cu - \cw]^2\right) + 2(\ru - \rw)\left( u-w\, -\, [\ru - \rw] \right) \\
	&\; =\;\quad \tfrac 12\left([\hu + \hw]^2 + [\cu + \cw]^2\right) + 2(\ru + \rw)\left( [u - \ru] + [w - \rw] \right) \\
	&\qquad\; -\; \tfrac 12\left([\hu - \hw]^2 + [\cu - \cw]^2\right) + 2(\ru - \rw)\left( [u - \ru] - [w - \rw] \right) \\
	&\; =\; 4\cdot\left[ \tfrac 12\left(\cu\cw + \hu\hw\right) + \rw  [  u - \ru ] +  \ru [ w   - \rw ] \right]. 
\end{align*}
\end{proof}


Hereafter we will denote by \(\lVert\cdot\rVert\equiv\lVert\cdot\rVert_\infty\) the infinity norm. Due to the norm equivalence in finite dimensional spaces all inequalities to be derived take the same form in other norms, provided the constants are adjusted accordingly. The infinity norm is particularly convenient, since we can then prove the following result for the general vector case $m>1$ by considering the absolute values of the individual components $f = F_i$ for $i= 1 \ldots m$. Moreover, we will make use of the Appolonius identity in the following proposition:  

\begin{proposition}
\label{prop:approx}
Suppose $\tilde x, \cx, \hx, \cy, \hy, \cz, \hz \in \R^n$ are restricted to a sufficently small closed convex neighboorhod $K \subset \R^n$ where the evalution procedure  for $ F : \R^n \mapsto \R^m $ is well defined. Then there are Lipschitz constants $\beta_F$ and  $\gamma_F$ such that we have
\begin{enumerate}[(i)]
\item Lipschitz continuity of function, tangent and secant models: 
  \begin{alignat*}{4}
   \|F(x)- F(\tilde x)\| \; &\leq \;
 \beta_F \|x-\tilde x\|\; &\text{ for }&x,\tilde x \in K,  \\
\max \left (
 \|\lozenge_{\cx}^{\hx}  F(x) -  \lozenge_{\cx}^{\hx}  F( \tilde x)\| ,
 \|\lozenge_{\rx}  F(x) -  \lozenge_{\rx} F( \tilde x)\| \right )    \; &\leq \;
 \beta_F \|x-\tilde x\|\;& \text{ for }&x,\tilde x \in \R^n.
\end{alignat*}
 The constant $\beta_F$  can be defined by the recurrences  $\beta_v= \beta_u+ \beta_w$ if $v=u+v$, $\beta_v= \beta_u$ if $v = |u|$ and 
$$\beta_v = \beta_u L_K(\varphi)  \quad  \mbox{if} \quad  v =\varphi(u) \quad \mbox{with} \quad L_K(\varphi) \equiv \max_{x \in K} |\varphi^\prime (u(x)) |\; . $$   
\item Error between function and secant or tangent model:   
\begin{eqnarray*}
  \|F(x)-   \lozenge_{\cx}^{\hx}  F(x) \| &\; \leq \; & \tfrac{1}{2}  \gamma_F \|x - \cx\|\|x - \hx\|  \\[0.2cm]  
    \|F(x)-   \lozenge_{\rx} F(x) \|  &\; \leq \; &\tfrac{1}{2}  \gamma_F \|x -
    \rx\|^2\; ,
    \end{eqnarray*}
    where $x\in K$.
The constant $\gamma_F$ can be defined using the recurrences $\gamma_v= \gamma_u+ \gamma_w$ if $v=u+w$, $\gamma_v= \gamma_u$ if $v = |u|$ and 
$$\gamma_v =  L_K(\varphi) \gamma_u  + L_K(\varphi^\prime) \beta_u^2  \quad \mbox{if} \quad  v =\varphi(u)
\quad \mbox{with} \quad L_K(\varphi^\prime) \equiv \max_{x \in K} |\varphi^{\prime\prime} (u(x)) |\; . $$  
   
\item Lipschitz continuity of secant and tangent model: 
\begin{align*}
 \|   \lozenge_{\cz}^{\hz}  F(x) - \lozenge_{\cy}^{\hy}  F(x) \| &\; \leq \; \gamma_F \;
\max \left [\| \hz-\hy \| \max ( \|x-\cy\|,  \|x-\cz\| ),  \right .  \\
&  \hspace*{2.1cm}  \left . \| \cz-\cy \| \max ( \|x-\hy\|,  \|x-\hz\| ) \right ]  \\[0.2cm]
 \|   \lozenge_{\rz}  F(x) - \lozenge_{\ry}  F(x) \|  &\; \;  \, \leq \quad 
 \gamma_F \quad \;  \,  \|\rz-\ry\|  \max (\|x-\ry\|,  \|x-\rz\|  )\;,   &
\end{align*}
where $x\in\mathbb R^n$.

\item Lipschitz continuity of the incremental part:
Let $x\in K$. Abbreviating $\Delta y = \hy-\cy$ and  $\Delta z = \hz-\cz$ we obtain in the secant case
\begin{align*}
\; & \|   \Delta_{\cz}^{\hz}  F(\Delta x) - \Delta_{\cy}^{\hy}  F(\Delta x) \|  \\
\leq \; &   2\beta_F \| \rz -\ry\| +  
  \tfrac{1}{2}\gamma_F  \left ( \| \rz-\ry \| + \max  (\| \Delta y\|, \|\Delta
  z\| )\right )^2  \\
  +  &  \gamma_F  \left  (\| \rz-\ry \| +  \tfrac{1}{2}( \| \Delta y\|+ \|\Delta
  z\|) \right ) \|\Delta x\|\; ,
  \end{align*}
  which reduces in the tangent case to 
  \[
  \| \Delta_{\rz} F( \Delta x) - \Delta_{\ry} F(\Delta x) \|  
\;  \leq \;  2\beta_F \| \rz -\ry\| + \tfrac{1}{2}\gamma_F\lVert \rz -
\ry\rVert^2 + \gamma_F  \lVert \rz - \ry \rVert\|\Delta x\|\; .  \]

\end{enumerate} 
Here the second bounds applying to the tangent model are always specializations of the previous ones for the secant model. 
\end{proposition} 

\begin{proof}
Since otherwise the bounds can be applied componentwise we may assume without loss of generality that $F$ is a scalar function $f$ and the norm in the range is simply the absolute value $|\cdot |$.   
The proof proceeds by induction on the intermediate quantities $v$ in the computational graph of $f$. We will define the constants 
$\beta_v$ and $\gamma_v$ recursively on the basis of the Lipschitz constants of the elemental functions and their derivatives:

	\emph{Variable initialization:} The initialization of independant variables represent the minimal nodes and all assertions are tivially true with the constants $\beta_{x_i} = 1$ and $\gamma_{x_i} =0$.
	
	\emph{Smooth univariate operation:} Let \(v = \varphi(u)\), \(\varphi \in \tilde \Phi\), be some elemental function in the computational graph of \(f\):
	\begin{align*}
		\lVert \varphi(u) - \varphi(\tilde u) \rVert &\le L_K(\varphi) \lVert u - \tilde u \rVert \le \beta_u L_K(\varphi)\lVert x - \tilde x \rVert \\
		\text{and}\quad\lVert \lozenge_{\cu}^{\hu} \varphi(u) - \lozenge_{\cu}^{\hu} \varphi(\tilde u) \rVert &\le L_K(\lozenge_{\cu}^{\hu} \varphi) \lVert u - \tilde u \rVert \le \beta_u L_K(\lozenge_{\cu}^{\hu} \varphi)\lVert x - \tilde x \rVert.
	\end{align*}
	The first inequality holds due to the Lipschitz continuity of \(\varphi\) and \(\lozenge_{\cu}^{\hu} \varphi\). 
	The latter inequality is the induction hypothesis and since \(L_K(\lozenge_{\cu}^{\hu} \varphi) \le L_K(\varphi)\) holds by the mean value theorem, we may set \(\beta_v \equiv \beta_u L_K(\varphi)\).
	
	\emph{Absolute value function and sum:} The absolute value function $v = \rabs(u)$ naturally maintains the Lipschitz constant and the addition summates them. 

Thus we have established $(i)$ for all cases.
  
Now let us consider the approximation property $(ii)$. For additions $v= u+w$ we may set $\gamma_v = \gamma_u+ \gamma_w$ and then have by the triangle inequality  
\begin{align*}
& |v(x)-   \lozenge_{\cx}^{\hx}  v(x) | \; \leq \;  |u(x)-   \lozenge_{\cx}^{\hx}  u(x) | +  |w(x)-   \lozenge_{\cx}^{\hx}  w(x)  | \\
& \; \leq  \;  \tfrac{1}{2} \gamma_u (\|x - \cx\|\|x - \hx\|) +   \tfrac{1}{2}  \gamma_w (\|x - \cx\|\|x - \hx\|)  \; = \;   \tfrac{1}{2} \gamma_v (\|x - \cx\|\|x - \hx\|) \,.  
 \end{align*} 
 For the absolute value function $v = \rabs(u)$ we may also set $\gamma_v = \gamma_u$, since 
 $$ |v(x)-   \lozenge_{\cx}^{\hx}  v(x) |  = | |u| -   | \lozenge_{\cx}^{\hx}  u(x) |  |  \leq   | u -    \lozenge_{\cx}^{\hx}  u(x) | 
 \leq  \tfrac{1}{2} \gamma_v (\|x - \cx\|\|x - \hx\|) \,. $$
For the univariate functions $v= \varphi(u)$ we have with $ \tilde  u \equiv \lozenge_{\cx}^{\hx} u(x)$
\begin{align*}
|v(x)-   \lozenge_{\cx}^{\hx}  v(x) | & \leq  |  \varphi(u(x)) - \varphi( \tilde u) |    +  
| \varphi( \tilde u) -  \lozenge_{\cu}^{\hu} \varphi(\tilde u)  | \,.
\end{align*}
By the mean value theorem and the induction hypothesis, the first term is bounded by 
$$   L_K(\varphi) |u(x) - \lozenge_{\cx}^{\hx}  u(x) | \leq   \tfrac{1}{2} L_K(\varphi) \gamma_u (\|x - \cx\|\|x - \hx\|) \,.  $$
The second term represents the error in the Hermite interpolation of $\varphi$ between $\cu$ and $\hu$. 
 With $L_K(\varphi^\prime)$ a Lipschitz constant of $\varphi^\prime$ on $u(K)$ it is bounded by 
 $$ \tfrac{1}{2} L_K(\varphi^\prime) |\tilde u - \cu|  |\tilde u - \hu|   \; \leq \; \tfrac{1}{2} L_K(\varphi^\prime)   \beta_u^2 \, (\|x - \cx\|\|x - \hx\|) \,, $$ 
 where the last bound follows from the fact that according to $(i)$ the approximation $\lozenge_{\cx}^{\hx} u(x)$ has the 
 Lipschitz constant $\beta_u$ and takes on at $\cx$ and $\hx$ the values $\cu$ and $\hu$.  Hence, we have shown that (ii) 
 holds indeed with 
 \begin{equation}
  \gamma_v \; = \;  L_K(\varphi)   \gamma_u \, + \, L_K(\varphi^\prime)   \beta_u^2 \,.
  \label{gammarecur}
\end{equation} 
Next we want to prove Lipschitz continuity of the model as stated in $(iii)$. Again, 
we find for additions and the abs function that the assertion is almost trivial 
with the constants $\gamma$ either being summated or just passed on. The challenge 
is once more the induction through the nonlinear univariates $v = \varphi(u)$.
To limit the notational complexity we will connect the two point pairs at 
the $u$ level by straight lines setting 
\begin{eqnarray*}
   \cu \equiv  \cu(t) \equiv   u(\cy)(1-t) + t u(\cz)  & \; \Rightarrow \; &  \partial \cu(t)/\partial t =  \Delta \cu  \equiv  u(\cz) - u(\cy) \,, \\
  \hu \equiv  \hu(t) \equiv   u(\hy)(1-t) + t u(\hz)   &\; \Rightarrow \; & \partial \hu(t)/\partial t =  \Delta \hu  \equiv  u(\hz) - u(\hy) \,, \\
  \tilde u \equiv  \tilde u(t) \equiv  \lozenge_{\cy}^{\hy} u(x) (1-t) + t \lozenge_{\cz}^{\hz} u(x) 
& \; \Rightarrow \; &  \partial  \tilde u / \partial t =  \Delta  \tilde u \equiv   \lozenge_{\cz}^{\hz} u(x) - \lozenge_{\cy}^{\hy}  u(x) \,.
\end{eqnarray*}
Here $x$ is fixed and we assume as induction hypothesis  that 
 \begin{equation}  \| \Delta \tilde u \| \; \leq \;    \gamma_u \max
 \bigl [\| \hz-\hy \| \max ( \|x-\cy\|,  \|x-\cz\| )  ,  \| \cz-\cy \| \max ( \|x-\hy\|,  \|x-\hz\| ) \bigr] \,. \label{induction}
\end{equation} 
To connect the piecewise linearizations of $v$ define
\begin{align*}
\tilde v(t)&\equiv\tfrac{1}{2} \left [ \varphi(\cu(t)) + \varphi(\hu(t))  \right  ]  +   \frac{\varphi(\cu(t)) -  \varphi(\hu(t))}{\cu(t)-\hu(t)} 
 \left[  \tilde u(t)   -   \tfrac{1}{2} \left (\cu(t) + \hu(t) \right ) \right ] 
\end{align*}
The quantity we want to find a bound for is 
\(
\Delta\tilde v \equiv \tilde v(1)-\tilde v(0) 
      = \lozenge_{\cz}^{\hz}  v(x)-\lozenge_{\cy}^{\hy}  v(x).  
\)
By the mean value theorem we find some $\bar t \in [0,1]$ where 
\begin{align*}
\Delta \tilde v &= \tilde v(1)-\tilde v(0)
= \left.\frac{\partial\tilde v(t)}{\partial t}  \right|_{t=\bar t}   \\ 
& =  \tfrac{1}{2} \left [ \varphi^\prime (\cu) \Delta \cu  + \varphi^\prime (\hu)  \Delta \hu   \right  ]  +  \frac{\varphi(\cu) -  \varphi(\hu)}{\cu-\hu} 
\left[  \Delta \tilde u   -   \tfrac{1}{2} \left (\Delta \cu + \Delta \hu \right ) \right ]   \\
& +   \left \{  \frac{ \varphi^\prime (\cu) \Delta \cu  - \varphi^\prime (\hu)  \Delta \hu }{\cu-\hu}  -  
 \frac{(\varphi(\cu) -  \varphi(\hu))( \Delta \cu - \Delta \hu )}{(\cu-\hu)^2}    \right \}    \left[  \tilde u   -   \tfrac{1}{2} \left (\cu + \hu \right ) \right ] .  
 \end{align*}
The functions $\cu,\hu,\tilde u$ here and following are to be read as evaluated at $t=\bar t$.
Now introduce \(\bar u\) as the mean value of \(\cu\) and \(\hu\) at which the difference quotient of $\varphi$ over the intervening interval is equal to its derivative. This yields:
\begin{align*}
 \Delta \tilde v & =  \tfrac{1}{2} \left [ \varphi^\prime (\cu) \Delta \cu  + \varphi^\prime (\hu)  \Delta \hu   \right  ]  +  \varphi^\prime{(\bar u)}  
\left[  \Delta \tilde u   -   \tfrac{1}{2} \left (\Delta \cu + \Delta \hu \right ) \right ]   \\
 & +   \left \{  \frac{    \varphi^\prime (\cu)  \Delta \cu   +  
 \varphi^\prime (\bar u ) ( \Delta \hu - \Delta \cu ) - \varphi^\prime (\hu) \Delta \hu }{\cu-\hu}    \right \}    \left[  \tilde u   -   \tfrac{1}{2} \left (\cu + \hu \right ) \right ]  \\
  & =  \tfrac{1}{2} \left \{ \frac{(\varphi^\prime (\cu) -  \varphi^\prime{(\bar u)} ) \Delta \cu  + (\varphi^\prime (\hu)-\varphi^\prime{(\bar u)}) \Delta \hu }{\cu - \hu}  \right  \}   (\cu - \hu)  +  \varphi^\prime{(\bar u)}   \Delta \tilde u    \\
 & +  \tfrac{1}{2} \left \{  \frac{ (\varphi^\prime (\cu) - \varphi^\prime (\bar u) )  \Delta \cu  - (\varphi^\prime (\hu)   -  
 \varphi^\prime (\bar u ) ) \Delta \hu  }{\cu-\hu}    \right \}    \left[  (\tilde u   -  \cu) + ( \tilde u   -\hu  ) \right ]  \\[0.2cm]
  & = \frac{\varphi^\prime(\bar{u})-\varphi^\prime(\hat u)}{\check{u}-\hat{u}}\Delta\hat{u}(\tilde{u}-\check{u})-\frac{\varphi^\prime(\bar{u})-\varphi^\prime(\check{u})}{\check{u}-\bar{u}}\Delta\check{u}(\tilde{u}-\hat{u}) + \varphi^\prime(\bar{u})\Delta\tilde{u} \,. 
\end{align*} 
The two quotients are bounded according to 
\begin{eqnarray*}
    \left |  \frac{\varphi^\prime(\bar{u})-\varphi^\prime(\check u)}{\check{u}-\hat{u}} \right | + \left |  \frac{\varphi^\prime(\bar{u})-\varphi^\prime(\hat u)}{\check{u}-\hat{u}} 
    \right | & \! \!  \leq \! \!   & 
 \left |  \frac{\varphi^\prime(\check{u})-\varphi^\prime(\bar u)}{\check{u}-\bar{u}}  \right |  \left |\frac{ \bar{u}-\check{u} }{\check{u}-\hat{u}} \right | + \left |  \frac{\varphi^\prime(\hat{u})-\varphi^\prime(\bar u)}{\hat{u}-\bar{u}}  \right |  \left |\frac{ \bar{u}-\hat{u} }{\check{u}-\hat{u}} \right |   \\ & \! \! \leq \! \!   & L_K(\varphi^\prime)  \left ( \left |\frac{ \bar{u}-\check{u} }{\check{u}-\hat{u}} \right | + \left |\frac{ \bar{u}-\hat{u} }{\check{u}-\hat{u}} \right | \right )  =   L_K(\varphi^\prime) \,,
  \end{eqnarray*}
  where the last equality follows from $\bar u$ being between $\cu$ and $\hu$. Hence, we find that 
\begin{equation} \left |  \Delta \tilde v  \right |  \;  \leq \;    L_K(\varphi^\prime) 
\max \left (|\Delta\check{u}||\tilde{u}-\hat{u}|+| \Delta\hat{u}||\tilde{u}-\check{u}| \right ) + L_K(\varphi) |\Delta\tilde{u} |  \,.
\label{firstbound}
\end{equation} 
The factors in the middle are easily bounded by 
   \[    
  | \Delta\check{u} | = | u(\check{z})-u(\check{y})| \leq \beta_u\left \|\cz-\cy\right \| 
  \quad \mbox{and} \quad
  | \Delta\hat{u} | = | u(\hat{z})-u(\hat{y})| \leq \beta_u\left \|\hat{z}-\hat{y}\right \|  
  \,.  \]
That leaves us with the second factors, which are linear in $t$ such that 
\begin{eqnarray*}
|\tilde u(t)- \cu(t) |  & \leq  & \max( |\tilde u(0)- \cu(0) |, |\tilde u(1)- \cu(1) |)   \\
& = & \max(| \lozenge_{\cy}^{\hy} u(x) - u(\cy) |, |\lozenge_{\cz}^{\hz} u(x) - u(\cz)| ) \\ 
&= &  \max(| \lozenge_{\cy}^{\hy} u(x) -   \lozenge_{\cy}^{\hy} u(\cy) |, |\lozenge_{\cz}^{\hz} u(x) -  \lozenge_{\cz}^{\hz} u(\cz)| )  \\
& \leq &    \beta_u \max(\| x - \cy\|, \| x - \cz\| )  \,,
\end{eqnarray*}
where the last inequality follows from (i). Analogously, we can derive
$$ |\tilde u(t)- \hu(t) | \leq  \beta_u \max(\| x - \hy\|, \| x - \hz\| )\;. $$
Substituting this into  \eqref{firstbound} we get
$$ \left |  \Delta \tilde v  \right |  \;  \leq \; 
\gamma_v  \max(\left \|\cz-\cy\right \|
\| x - \hy\|, \left \|\cz-\cy\right \| \| x - \hz\|, \left \|\hat{z}-\hat{y}\right \| 
\| x - \cy\|, \left \|\hat{z}-\hat{y}\right \|  
\| x - \cz\|   )\, ,   $$
with $\gamma_v \equiv  L_K(\varphi) \gamma_u + L_K(\varphi^\prime ) \beta_u^2 $, which completes the proof of $(iii)$. 
Finally, we have to prove $(iv)$, which gives a bound on the increment part only. Setting $\xi\equiv\tfrac{1}{2}(\rz+\ry)+\Delta x$ one gets, with 
the results already proved and a few triangle inequalities, that 
\begin{align*}
&\lVert\Delta_{\cz}^{\hz} F(\Delta x)-\Delta_{\cy}^{\hy} F(\Delta x)\rVert \\
=\;&\lVert\lozenge_{\cz}^{\hz}F(\rz+\Delta x)-F(\rz)-(\lozenge_{\cy}^{\hy}F(\ry+\Delta x)-F(\ry))\rVert \\
=\;&\lVert \lozenge_{\cz}^{\hz}F(\xi)-F(\rz)+\lozenge_{\cz}^{\hz}F(\rz+\Delta x) - \lozenge_{\cz}^{\hz}F(\xi) \\&- \lozenge_{\cy}^{\hy}F(\xi)+F(\ry)-\lozenge_{\cy}^{\hy}F(\ry+\Delta x) + \lozenge_{\cy}^{\hy}F(\xi)\rVert \\
\leq \;&\beta_F\lVert\rz-\ry\rVert +\beta_F\lVert\rz+\Delta x-\tfrac{1}{2}(\rz+\ry)-\Delta x\rVert+\beta_F\lVert\ry+\Delta x-\tfrac{1}{2}(\rz+\ry)-\Delta x\rVert\\
&+\gamma_F\max[\lVert\cz-\cy\rVert\max(\lVert\xi-\hz\rVert,\lVert\xi-\hy\rVert), \lVert\hz-\hy\rVert\max(\lVert\xi-\cz\rVert,\lVert\xi-\cy\rVert)]\\
=\;&2\beta_F\lVert\rz-\ry\rVert+\gamma_F\max[\lVert\cz-\cy\rVert\max(\lVert\xi-\hz\rVert,\lVert\xi-\hy\rVert),\lVert\hz-\hy\rVert\max(\lVert\xi-\cz\rVert,\lVert\xi-\cy\rVert)] \,.
\end{align*} 
Now, since for example 
$$\| \tfrac{1}{2} (\ry +\rz) - \cz\| = \| \tfrac{1}{2} (\ry - \rz) - \cz + \tfrac{1}{2} 
(\cz+\hz) \| \leq   \tfrac{1}{2}  (\| \ry - \rz \| +  \|\Delta z\|)\, , $$
both inner maxima can be bounded by the same expression, namely 
$$ \|\Delta x\| +    \tfrac{1}{2} (\|\rz-\ry\| +\max( \|\Delta z\|, \|\Delta y\| ))\, , $$ 
so that we obtain the upper bound 
\begin{align*} 
 \; & \lVert\Delta_{\cz}^{\hz} F(\Delta x)-\Delta_{\cy}^{\hy} F(\Delta x)\rVert \\
 \leq \; &  2\beta_F\lVert\rz-\ry\rVert + \gamma_F  \max(\lVert\cz-\cy\rVert, \lVert\hz-\hy\rVert) \left [ \|\Delta x\| +    \tfrac{1}{2} (\|\rz-\ry\| +\max( \|\Delta z\|, \|\Delta y\| ))\right ] \,.
\end{align*}
Finally, we can also bound 
$$\| \cy - \cz  \|   \; = \| \cy - \ry - \cz + \rz + \ry - \rz\| \leq \|\ry - \rz\| +   \tfrac{1}{2}(\|\Delta y\|+ \|\Delta x\|)  \,,  $$
which yields the assertion after some elementary modifications. 
>From the secant result we can easily get the bound for the tangent model by
setting \(\cz=\hz=\rz\) and \(\cy=\hy=\ry\).
\end{proof} 
{}
{}
As one can see by setting $\cy=x=\cz$ the assertion $(ii)$ almost follows from $(iii)$, except that  a factor of 2 is lost in the constants. 
The proposition also states  that the values of $F$  at $\cx$ and $\hx$ are reproduced exactly by our approximation as one would expect from a secant approximation. 
This property clearly nails down the piecewise linearization rules \eqref{:absdiff} and \eqref{:linphi} with \eqref{:cdef} for all  univariate functions.  
Also, there is no doubt that addition and subtraction should be linearized according to \eqref{:linadd} and that  multiplications $v_i =c \,v_j$ by constants
$c$ should yield the differentiated version $\Delta v_i= c\, \Delta v_j$, which is a special case of  \eqref{:linmult}. 
For general multiplications  $v_i = v_j \ast v_k$ the two values $\cv_i=  \cv_j \ast \cv_k$
and $\hv_i= \hv_j \ast \hv_k$ could also be interpolated by  linear functions other than the one defined by \eqref{:linmult}.
However, we currently see no possible gain in that flexibility,
and maintaing the usual product rule form seems rather attractive.

\subsection*{Stable Surjectivity}

Recall that a continuous function $F:\mathbb R^n\rightarrow\mathbb R^m$ is called proper if the preimage of every compact set is compact. 
One can easily see that any piecewise linear function is proper if and only if it maps no affine ray $\{a+\lambda b\; :\; \lambda\ge 0 \}, (a, b\in\mathbb R^n),$ to a point. 
This is trivially the case if $n=m$ and the Jacobians of all selection functions of $F$ are invertible.
In \cite{degree} it was proved that if a piecewise linear function $F:\mathbb R^n\rightarrow\mathbb R^n$ is proper, there exists a $d\in\mathbb Z$ such that for all regular values
$y$ of $F$, i.e. all values $y$ such that the Jacobian at all preimages of $y$ exists and is invertible [which is trivially the case if $F^{-1}(y)=\emptyset$], it holds
$$d\ =\ \sum_{x\in F^{-1}(y)}\operatorname{sign}[\operatorname{det}(D_xF)]\ =:\ \operatorname{deg(F)}\; .$$
Hence, any regular value $y$ of a proper piecewise linear function $F:\mathbb R^n\rightarrow\mathbb R^n$ has at least one preimage if $\operatorname{deg}(F)\ne 0$. This implies surjectivity of $F$ by the closedness of piecewise linear functions (cf. \cite{scholtespl}) and the well known fact that regular values lie dense in the range. We call $\operatorname{deg}(F)$ the degree of $F$.

\begin{lemma}(\cite[Cor. 5.2]{degree}) \label{degDistance}
 Assume a piecewise linear function $F:\R^n\rightarrow\R^n$ is composed of $k$ affine functions $F_i$ with invertible linear parts $A_i$ and define 
 $$\rho_F\ :=\ \min\left\{\frac{1}{\Vert A^{-1}_1\Vert},\dots, \frac{1}{\Vert A^{-1}_k\Vert}\right\}\; .$$
 Moreover, let $0< \varepsilon<\rho_F$ and $l\ge 0$. Then all piecewise linear functions $G:\R^n\rightarrow\R^n$ with
 $$\Vert F(x)-G(x)\Vert\ \le\ \varepsilon\Vert x\Vert +l\qquad\forall x\in\R^n$$
are proper and their degree equals that of $F$.  
\end{lemma}
The latter statement enables us to prove another stability result:
\begin{proposition}\label{degConstRadius}
Let \(F\in\text{span}(\Phi)\) where \(F:\R^n\to\R^n\) and assume a piecewise linearization $\lozenge_{\rx} F:\R^n\rightarrow\R^n$ is composed of $k$ affine functions $F_i$ with invertible linear parts $A_i$. Moreover, define $\rho_F$ as in Lemma \ref{degDistance}. Then for all tangent and secant mode piecewise linearizations with development points in the ball $B(\rx, \rho_F/\gamma_F)$, where $\gamma_F$ is defined as in Proposition \ref{prop:approx}, the mapping degree is well defined and equals that of $\lozenge_{\rx} F$. 
\end{proposition}
\begin{proof}
Just plug Proposition \ref{prop:approx}. $(iii)$ into Lemma \ref{degDistance}.
\end{proof}
The fact that the mapping degree of a coherently oriented piecewise linear function cannot be zero yields a rather pointed statement:
\begin{corollary}
In the situation of Proposition \ref{degConstRadius}, if $\lozenge_{\rx} F$ is coherently oriented, then all tangent and secant mode piecewise linearizations with development points in the ball $B(\rx, \rho_F/\gamma_F)$ are surjective. 
\end{corollary}

\section{Generalized Newton Methods by Piecewise Linearization}

We will proceed by proposing and analyzing a possible application of the 
tangent and secant approximations developed in the previous sections. 
For this we present generalized Newton's methods for 
composite piecewise smooth functions \(F:\R^n\to\R^n\) based on the piecewise linear 
approximations, both for the tangent and the secant mode. The merit of these methods 
is the fact that they impose no strong differentiability requirements but require 
only piecewise differentiability at the root in question. 

\begin{definition}[Newton operator]\label{def1}
Let $F\in {\rm span}(\Phi)$ and $x^\ast$ be an isolated root of $F$ 
in an open neighborhood ${\cal D}$. The Newton step for $F$
is definable on ${\cal D}$ in tangent or secant mode if the piecewise linear equation
$\lozenge_{\rx} F(x)=0$ resp. $\lozenge_{\cx}^{\hx} F(x)=0$ has at least 
one root for all $\rx\in \mathcal D$ resp. $\cx,\hx\in {\cal D}$.

Then the Newton operator is defined in tangent mode as 
$$
N(\rx)=\arg\min\{\|x-\rx\|:\lozenge_{\rx} F(x)=0\}
$$
and in secant mode as
$$
N(\hx, \cx) = \arg\min\{\|x-\tfrac 12(\hx + \cx)\|:\lozenge_{\cx}^{\hx} F(x)=0\}\; .
$$
\end{definition}
Definition \ref{def1} is the minimal assumption under which we can conclude, using the approximation and Lipschitz continuity results of Section \ref{secApprox}, that the iteration 
\begin{align}\label{NewtonDef}
x_{k+1} &= N(x_k) \quad \text{(tangent)} & x_{k+1} &= N(x_k, x_{k-1}) \quad \text{(secant)}
\end{align}
 converges locally to an isolated root $x^\ast$.
 Assuming that the piecewise linearization $\lozenge_{x^\ast}F$ 
is bijective on a ball ${\cal D}=B(x^\ast,\bar R)$ we will proceed to show that close to $x^\ast$ 
\begin{itemize}
 \item the Newton step can be defined
 \item the Newton step stays close to $x^\ast$
 \item the tangent mode Newton method converges quadratically
 and the secant mode Newton method converges with the golden mean as order.
\end{itemize}

%
%
%
We know that the assumption of bijectivity on a ball can be relaxed through the use of degree theory, but this requires a significant technical effort for which we refer to subsequent works. Note that the minimal assumption for the local convergence of semismooth Newton \cite{semismooth} is similar, namely, that there always exists a generalized Jacobian $J(x)$ that has a uniformly bounded inverse over all $x \in B(x^\ast, \bar\rho)$. 

For most of the following results up to the last it is sufficient to consider 
the secant mode piecewise linear approximation of $F$, as results for the tangent 
mode can be obtained by setting $\cx=\hx=\rx$.

The general bound for the difference of piecewise linearizations with different
basis points 
$\cz,\hz$ and $\cy=\hy=\ry$
can be refined to a closer bound in the case that one is to express that bound only in 
distances to $\ry$.
\begin{lemma}\label{refinement}
  Let the second order constant $\gamma_F$ be valid on some convex set $U$ and
$x,\ry,\cz,\hz\in U$.
  Then 
  $$
  \|\lozenge_{\cz}^{\hz} F(x)-\lozenge_{\ry} F(x)\|
  \le \gamma_F
\left[ \max(\|\hz-\ry\|,\|\cz-\ry\|)\cdot\|x-\ry\|\; +\; \frac12\|\hz-\ry\|\,\|\cz-\ry\| \right].
  $$
\end{lemma}
\begin{proof}
 Select some $N\in \mathbb N$ and consider the subdivision of the segments $[\cz,\ry]$ and
$[\hz,\ry]$
 by $\cu_k=\ry+\frac kN (\cz-\ry)$ and $\hu_k=\ry+\frac kN (\hz-\ry)$.
 Then by Prop. \ref{prop:approx}.$(iii)$ 
 \begin{align*}
  \|\lozenge_{\cu_{k+1}}^{\hu_{k+1}} F(x)-\lozenge_{\cu_{k}}^{\hu_{k}} F(x)\|
  &\le \gamma_F\max\left[\;\begin{gathered}
    \|\hu_{k+1}-\hu_k\|\,\max(\|x-\cu_{k+1}\|,\|x-\cu_k\|)\\
    \|\cu_{k+1}-\cu_k\|\,\max(\|x-\hu_{k+1}\|,\|x-\hu_k\|)            
    \end{gathered}\;\right]
  \notag\\
  &\le\gamma_F \max\left[\;\begin{gathered}
    \tfrac1N\|\hz-\ry\|\,\bigl(\|x-\ry\|+\tfrac{k+1}N\|\cz-\ry\|\bigr)\\
    \tfrac1N\|\cz-\ry\|\,\bigl(\|x-\ry\|+\tfrac{k+1}N\|\hz-\ry\|\bigr)            
    \end{gathered}\;\right]
  \notag\\
  &=\gamma_F\big[\tfrac1N\max(\|\hz-\ry\|,\|\cz-\ry\|)\|x-\ry\|+\tfrac{k+1}{N^2}\|\hz-\ry\|\|\cz-\ry\|\big]\; .
 \end{align*}
 Summation and limit $N\to\infty$ results in the claim.
\end{proof}

We assume hereafter that  $\lozenge_{x^\ast}F$ 
is bijective on a ball ${\cal D}=B(x^\ast,\bar R)$, where $x^\ast$ is an isolated root of $F$, which implies that it is also coherently oriented
and metrically
regular on the latter. This means that there exists a constant $c>0$ such that 
for all $x\in {\cal D}$, $y\in\lozenge_{x^\ast}F({\cal D})$
$$
\left\|x-\left(\lozenge_{x^\ast}F\mid_{\cal
D}\right)^{-1}\left(y\right)\right\|\le c\left\|\lozenge_{x^\ast}F(x)-y\right\|.
$$

\begin{lemma}[existence of roots of the PL approximation]\label{rootEx}
 If $R\le \min(\bar R,\frac1{c\gamma_F})$ and $\rho=\frac{R}3$ then for any $\cx,\hx\in
B(x^\ast,\rho)$
 the piecewise linearization $\lozenge_{\cx}^{\hx}F$ has a root in $B(x^\ast,R)$.
\end{lemma}
\begin{proof}
We intend to apply the Brouwer theorem to the fixed point operator
 \begin{equation}
  T(x) = x + x^\ast- \left(\lozenge_{x^\ast}F\mid_{\cal
D}\right)^{-1}\left(\lozenge_{\cx}^{\hx}F(x)\right).
 \end{equation}
 Any fixed point is then a root of $\lozenge_{\cx}^{\hx}F$.
 As $T$ is obviously continuous, we only need to show that $T$ maps $B(x^\ast,R)$ into
itself. Using metric regularity and Lemma ~\ref{refinement} we find
 \begin{align*}
  \|T(x)-x^\ast\|&=\left\|x-\left(\lozenge_{x^\ast}F\mid_{\cal
D}\right)^{-1}\left(\lozenge_{\cx}^{\hx}F(x)\right)\right\|
  \\
  &\le c\left\|\lozenge_{x^\ast}F(x)-\lozenge_{\cx}^{\hx}F(x)\right\|
  \\
  &\le c\gamma_F\big[\max(\|\cx-x^\ast\|,\|\hx-x^\ast\|)\,\|x-x^\ast\|\; +\;\tfrac12\|\cx-x^\ast\|\,\|\hx-x^\ast\|\big]
  \\
  &\le \frac{\rho}R\left(R+\frac12\rho\right)\le\frac{7}{18}R.
 \end{align*}
\end{proof}

\begin{lemma}[containment of the roots]\label{containment} Under the same assumptions 
any root of $\lozenge_{\cx}^{\hx}F$ in $B(x^\ast,R)$ is actually contained in
$B(x^\ast,\rho)$.
\end{lemma}
\begin{proof}
 Let $x$ be a root of $\lozenge_{\cx}^{\hx}F$. Then again using metric regularity and 
 Lemma~\ref{refinement} the distance to $x^*$ has the bound
 \begin{align*}
  \|x-x^\ast\|&\le
c\|\lozenge_{x^\ast}F(x)\|=c\|\lozenge_{x^\ast}F(x)-\lozenge_{\cx}^{\hx}F(x)\|
  \\
  &\le c\gamma_F
\big[\max(\|\cx-x^\ast\|,\|\hx-x^\ast\|)\,\|x-x^\ast\|\; +\;\tfrac12\|\cx-x^\ast\|\,\|\hx-x^\ast\|\big]
  \\
  &\le \frac13\,\|x-x^\ast\| + \frac{\|\cx-x^\ast\|\,\|\hx-x^\ast\|}{6\rho}
 \end{align*}
  so that
  \begin{equation}\label{superlinear}
  \|x-x^\ast\|\le \frac{\|\cx-x^\ast\|\,\|\hx-x^\ast\|}{4\rho}\le
\frac14\min(\|\cx-x^\ast\|,\|\hx-x^\ast\|)<\frac14\rho.
  \end{equation}
\end{proof}

\begin{corollary}[Newton iteration] Let $\lozenge_{x^\ast}F$ be bijective on a ball ${\cal D}=B(x^\ast,\bar R)$, 
where $x^\ast$ is an isolated root of $F\in {\rm span}(\Phi)$.
Then on $B(x^\ast,\rho)$, where $\rho$ is defined as in Lemma \ref{rootEx}, the Newton step is defined in
both the tangent and secant mode and maps back into $B(x^\ast,\rho)$. The thus definable 
Newton iteration converges at least linearly.
\end{corollary}
\begin{proof}
 As the next Newton iterate is among the roots of minimal distance to
 the basis point(s) of the linearization, one needs to ensure that any root
 that $\lozenge_{\cx}^{\hx}F$ may have outside the ball $B(x^\ast,R)$ has a 
 larger distance to $\{\cx,\hx\}$ than the root that is known to exist 
 inside $B(x^*,\frac14\rho)$. The distance from the basis points inside $B(x^*,\rho)$
 to the outside of $B(x^*,R)$ is at least $R-\rho=\frac23 R$. The distance from the 
 root of $\lozenge_{\cx}^{\hx}F$ inside $B(x^*,\frac14 R)$ to the basis points is 
 at most $\frac14\rho+\rho=\frac5{12}R$ and thus the smaller distance.
 
 By equation~\eqref{superlinear} of the last lemma we also see that the distance of
 the root of $\lozenge_{\cx}^{\hx}F$
 to $x^*$ is at most $\frac14$ the distance of the basis points to $x^*$, which implies 
 linear convergence.
\end{proof}

\begin{corollary}[convergence rates]
 Let $R=\min(\bar R, \frac1{c\gamma_F})$ and $\rho=\frac{R}3$. Then for all initial
points 
 $x_0\,(,x_1)\in B(x^\ast,\rho)$ the Newton iteration in tangent mode 
$$
x_{k+1} = N(x_k)
$$
 converges quadratically resp. in secant mode 
$$
x_{k+1} = N(x_k,x_{k-1})
$$
converges with order $\frac{1+\sqrt5}2$ 
towards $x^\ast$.
\end{corollary}
%
%
%
%
\begin{proof}
With the choice of the initial points and by Lemma~\ref{containment}, the full iteration sequence 
stays inside $B(x^*,\rho)$. 

\emph{Tangent mode:}
By replacing $x,\cx,\hx$ with $x_{j+1},x_j,x_j$  equation~\eqref{superlinear} in Lemma~\ref{containment} 
we get
 \begin{align*}
  \|x_{j+1}-x^*\|\le \frac1{4\rho}\|x_j-x^*\|^2
 \end{align*}
 which implies the quadratic convergence of the sequence $(x_j)$ towards $x^*$,
 \[
   \|x_j-x^*\|\le 4\rho\left(\frac{\|x_0-x^*\|}{4\rho}\right)^{2^j}.
 \]

\emph{Secant mode:}
Replacing $x,\cx,\hx$ with $x_{j+2},x_{j+1},x_j$ in equation~\eqref{superlinear} 
in Lemma~\ref{containment} one finds
 \begin{align*}
  \|x_{j+2}-x^*\|\le \frac1{4\rho}\|x_{j+1}-x^*\|\,\|x_j-x^*\|.
 \end{align*}
 As in the scalar secant method, this implies convergence with rate $\phi=\frac{1+\sqrt5}2$
 or more precisely
 \[
   \|x_j-x^*\|\le 4\rho\left(\frac{\|x_0-x^*\|}{4\rho}\right)^{F_{j-1}}\left(\frac{\|x_1-x^*\|}{4\rho}\right)^{F_j}
   \le 4\rho\left(\frac{\max(\|x_0-x^*\|,\|x_1-x^*\|)}{4\rho}\right)^{\phi^j}.
 \]
 where $(F_j)$ is the Fibonacci sequence with $F_0=0$, $F_1=1$ and $F_{j+1}\le \phi^j$.
\end{proof}
The inner iterations of \eqref{NewtonDef} require the 
solution of piecewise linear equations. 
However, these may possess several solutions \(\Delta x_j\), of which we must find one of minimal norm. 
So far, the solvers that we know and surveyed in \cite{mother, son, proceed, SGE} require at least 
coherent orientation to guarantee convergence. Hence, the actual implementation of successive 
piecewise linearization needs further study. 

\section{Singularity Free Implementation}
\label{appendix}

In contrast to the tangent mode of piecewise linearization, the secant mode involves two 
points of evaluation \(\check v_i, \hat v_i \in \R^n\), 
which means that its computational cost of the primal values are roughly twice that of the tangent mode.
These define a line segment 
\([\check v_i, \hat v_i] \equiv \{ \lambda \check v_i + (1-\lambda)\hat v_i \mid \lambda \in [0, 1]\}\) 
for any intermediate operation. The formal definition of the secant slope given in equation \eqref{:cdef} 
may cause numerically unstable divisions when the denominator gets small during the 
transition from secant to tangent mode,
e.g. when the secant mode Newton iteration scheme converges.\\
However, in this section we will provide \emph{singularity free} closed form expressions. 
Therefore an exception handling at \(\check v_i = \hat v_i\) will no longer be necessary. 
To that end we move from the line segment representation to a midpoint-radius based representation. 
Now let \(v = \varphi(u)\), where \(\varphi \in \{\sin, \exp, \dots\}\) is some elementary operation and
\begin{align*}
	(\check v_i, \hat v_i) = (\varphi(\check v_j), \varphi(\hat v_j))\,\,
	\mapsto\,\, (\mathring v_i, \mathrad v_i), \quad \text{where} \quad \mathring
	v_i = \frac{\check v_i + \hat v_i}2\;\text{ and }\; \mathrad v_i = -\frac{\check v_i -
	\hat v_i}2\,.
\end{align*}
We adopted the concept of representing intervals via midpoint and radius from interval arithmetic calculus 
(described in detail e.g. by Siegfrid Rump in \cite{Rump99}, 
or by G\"otz Alefeld and J\"urgen Herzberger in \cite{alefeld2012introductionInterval}). 
We remark though, that in the present setting the radius \(\mathrad v_i \in \R\) is allowed to become negative as well. 

Now one can rewrite the secant slope of differentiable functions to the new representation 
\[c_{ij} \equiv \frac{\check v_i - \hat v_i}{\check v_j - \hat
v_j} = \frac{\mathrad v_i}{\mathrad v_j}\,. \]
Using some algebraic manipulations one can find individual formulas for the aforementioned propagation rules of the secant mode:
\newcolumntype{C}{>{\displaystyle} c}
\begin{align*}
	\def\arraystretch{1.5}
	\begin{array}{|C|C|C|C|C|}
		\hline
		\text{binary operation} & \mathring v_i & \mathrad v_i & c_{ij} & c_{ik} \\
		\hline
		\hline
		v_i = v_j + v_k & \mathring v_j + \mathring v_k & \mathrad v_j + \mathrad v_k & 1
		& 1	\\
		v_i = v_j - v_k & \mathring v_j - \mathring v_k & \mathrad v_j - \mathrad v_k & 1
		& -1 \\
		v_i = v_j \cdot v_k & \mathring v_j \mathring v_k + \mathrad v_j\mathrad v_k &
		\mathrad v_j \mathring v_k + \mathring v_j \mathrad v_k & \mathring v_k &
		\mathring v_j \\
		v_i = \frac{v_j}{v_k} & \frac{\mathring v_j\mathring v_k - \mathrad v_j\mathrad
		v_k}{\mathring v_k^2 - \mathrad v_k^2} 
						      & \frac{\mathrad v_j \mathring v_k - \mathring v_j \mathrad
							    v_k}{\mathring v_k^2 - \mathrad v_k^2} 
							  & \frac 1{\mathring v_k} &
							    -\frac{\mathring v_j}{\mathring v_k^2 - \mathrad v_k^2}
							    \\
		\hline
	\end{array}
	\def\arraystretch{1.}
\end{align*}
Alternatively, we could represent \(v_i = v_j / v_k\) as an application of a multiplication on the reciprocal \(1/v_k\). Furthermore, we can represent the multiplication by the Appolonius identity as above. Moreover, for unary operations we get:
\begin{align*}
	\def\arraystretch{1.5}
	\begin{array}{|C|C|C|C|}
		\hline
		\text{unary operation} & \mathring v_i & \mathrad v_i & c_{ij} \\
		\hline
		\hline
		v_i = \sin(v_j) & \sin(\mathring v_j)\cos(\mathrad v_j) & \cos(\mathring
		v_j)\sin(\mathrad v_j) & \cos(\mathring v_j)\sinc(\mathrad v_j) \\
		v_i = \cos(v_j) & \cos(\mathring v_j)\cos(\mathrad v_j) & -\sin(\mathring
		v_j)\sin(\mathrad v_j) & -\sin(\mathring v_j)\sinc(\mathrad v_j) \\
		v_i = \exp(v_j) & \exp(\mathring v_j)\cosh(\mathrad v_j) & \exp(\mathring
		v_j)\sinh(\mathrad v_j) & \exp(\mathring v_j)\sinhc(\mathrad v_j) \\
		v_i = \log(v_j) & \frac 12\log(\mathring v_j^2 + \mathrad v_j^2) &
		\artanh\left( \frac{\mathrad v_j}{\mathring v_j} \right) & 
		\frac 1{\mathring v_j}\artanhc\left( \frac{\mathrad v_j}{\mathring v_j} \right)
		\\
		\hline
	\end{array}
	\def\arraystretch{1.}
\end{align*}
Note that by, e.g., \cite{mathworld_sinhc} \(\sinc(x)\) and \(\sinhc(x)\) (hyperbolic \(\sinc(x)\)) have regular Taylor expansions
\begin{align*}
	\sin(x) = x\cdot \sinc(x) = x\cdot \sum_{i=0}^\infty \frac{(-x^2)^n}{(2n+1)!},
	\quad \sinc(x) \equiv x\cdot \sinhc(x) \equiv
	x\cdot \sum_{i=0}^\infty
	\frac{x^{2n}}{(2n+1)!} \,.
\end{align*} 
We want to define \(\artanhc(x)\) similar to \(\sinhc(x)\) via
its Taylor expansion:
\[ \tanh^{-1}(x) = \artanh(x) = x\cdot
\artanhc(x) = x\cdot \sum_{i=0}^\infty \frac{x^{2n}}{2n+1} \]
and it can be implemented in a similar fashion as \(\sinhc(x)\) from the boost c++ libraries
(see \cite{BOOST_sinhc}).
For Root functions (\(v_i = \sqrt[c]{v_j}\)), general Powers (\(v_i = v_j^c\), or in a binary fashion \(v_i = v_j^{v_k}\)) and monomials (\(v_i = v_j^n\)) one can use the identity
\begin{align*}
	v_i = v_j^c = \exp(c\cdot \log(v_j)) \quad \text{or} \quad v_i = v_j^{v_k} =
	\exp(v_k\cdot \log(v_j))
\end{align*}
and apply the rules above. Of course, the base \(v_j > 0\) has to be positive, but there is a less restrictive alternative for monomials:
\begin{align*}
	\def\arraystretch{1.5}
	\begin{array}{|C|C|C|C|}
		\hline
		\text{monomials} & \mathring v_i & \mathrad v_i & c_{ij} \\
		\hline
		\hline
		\begin{array}{c} v_i = v_j^n \\ (n\text{ natural number}) \end{array} &
		\sum_{\substack{k\le n
		\\
		\text{even}}} \binom nk \mathring v_j^{n-k} \mathrad v_j^k 
			  & \sum_{\substack{k\le n \\ \text{odd}}}
			    \binom nk \mathring v_j^{n-k} \mathrad v_j^k & 
			    \sum_{\substack{k\le n \\ \text{odd}}}
			    \binom nk \mathring v_j^{n-k} \mathrad v_j^{k-1}
			    \\
		\begin{array}{c} v_i = v_j^2 \\ (\text{special case: square}) \end{array} &
		\mathring v_j^2 + \mathrad v_j^2 & 2\mathring v_j\mathrad v_j & 2\mathring v_j \\
		\hline
	\end{array}
	\def\arraystretch{1.}
\end{align*}

\begin{remark}[General approximation for unary operations]
Of course one can find a lot more singularity free formulas for the secant
slopes of other operations. Using a Taylor expansion approach one can
provide general approximation formulas for the triplet \(\mathring v_i, \mathrad
v_i\) and \(c_{ij}\) by:
\begin{align*}
	\mathring v_i &= \frac 12 \left( \sum_{k\ge 0}
	\frac{\varphi^{(k)}(\mathring v_j)}{k!}\mathrad v_j^k + \sum_{k\ge 0}
	\frac{\varphi^{(k)}(\mathring v_j)}{k!}(-\mathrad v_j)^k \right)
	= \sum_{\substack{k\ge 0 \\ k \text{ even}}}
	\frac{\varphi^{(k)}(\mathring v_j)}{k!} \mathrad v_j^k, \\
	\mathrad v_j &= \frac{\varphi(\mathring v_j + \mathrad v_j) - \varphi(\mathring v_j
	- \mathrad v_j)}2 = \sum_{\substack{k > 0 \\ k \text{ odd}}}
	\frac{\varphi^{(k)}(\mathring v_j)}{k!} \mathrad v_j^k,\quad c_{ij} =
	\sum_{\substack{k > 0 \\ k \text{ odd}}}
	\frac{\varphi^{(k)}(\mathring v_j)}{k!} \mathrad v_j^{k-1}.
%
\end{align*}
\end{remark}

\section{Numerical Example}

Consider the function
\begin{align*}
  F(x) = \left[\begin{array}{rr}
    \cos(\varphi(\angle x) - \angle x) & -\sin(\varphi(\angle x) - \angle x) \\
    \sin(\varphi(\angle x) - \angle x) & \cos(\varphi(\angle x) - \angle x)
  \end{array}\right]\cdot x - c\, ,
\end{align*}
where \(c=[1.001, 10.01]^\top\), and \(\angle\! x \in {[0, 2\pi[}\) is the angle of
\(x=(x_1,x_2)^\top\) in polar coordinate representation. Moreover, \( \varphi \), which is defined by
\[ \varphi(\psi) \equiv \psi + \frac{8}{5\pi}\psi^2 -
\frac{8}{5\pi^2}\psi^3 + \frac{2}{5\pi^3}\psi^4 \] 
maps \({[0, 2\pi[}\) strictly monotonically onto itself. Hence, $F(x)$ is bijective.
Furthermore, the function is differentiable everywhere except at the origin. 
There it is, just as the Euclidean norm, locally Lipschitz continuous and not
even piecewise differentiable in the sense of \cite{scholtespl}.

We investigated the behavior of the tangent and secant mode Newton on $F$ both with and without noise. 
That is, we investigated $F$ and $\tilde F$, where 
$$\tilde F(x) = F(x) + \frac{\sin(5000\cdot [x_1 +x_2])}{10^4} \; .$$
The secant mode was started with the initial values $\check x=[-3.7,-2.05]^\top$ and $\hat x= [7.0, 8.0]^\top$. 
The tangent mode was started with the mean value of the latter points.

We recall the well known formula for approximating the convergence rate numerically:
\[ \gamma \approx \frac{\log\left|\frac{x_{n+1} - x_n}{x_n - x_{n-1}}\right|}{\log\left|\frac{x_n - x_{n-1}}{x_{n-1} - x_{n-2}}\right|}\, . \]

The first table shows both modes' residuals in the iterations without noise. 
\begin{align*}
  \begin{array}{|c|c|c|}
    \hline
    \text{iteration} & \text{residual with tangent mode} & \text{res. with secant mode} \\
    \hline
    \hline
    0 & 13.3919956235 & 5.81435555868 \\
    1 & 5.65630249881 & 19.6157765738 \\
    2 & 3.39957297287 & 4.99712831052 \\
    3 & 0.00920821188601 & 1.2635478817 \\
    4 & 2.65403135735e-06 & 0.0882985228824 \\
    5 & 2.13162820728e-13 & 0.00192152171278 \\
    6 & 8.881784197e-15 & 5.19844852542e-06 \\
    \hline
    7 && 3.0607116841e-10 \\
    8 && 8.881784197e-16 \\
    \hline
    \hline
    \gamma & 2.07814399547 \approx 2 & 1.64753467681 \approx \frac{1+\sqrt 5}2 \\
    \hline
  \end{array}
\end{align*}

The next table shows the residuals of the iterations with noise.
\begin{align*}
  \begin{array}{|c|c|c|}
    \hline
    \text{iteration} & \text{residual with tangent mode} & \text{res. with secant mode} \\
    \hline
    \hline
    0 & 13.3920216719 & 5.81445152008 \\
    1 & 60.308012713 & 19.6157077846 \\
    2 & 81.8554424857 & 4.99709320593 \\
    3 & 8.54532016753 & 1.26352999658 \\
    4 & 5.69986744799 & 0.0883961689187 \\
    5 & 1.92933721639 & 0.00178884970844 \\
    6 & 0.425650504358 & 0.00013247524387 \\
    7 & 0.127098157087 & 3.54530493811e-05 \\
    8 & 0.0253173333077 & 2.27749388494e-06 \\
    9 & 0.00248354296218 & 3.65673925216e-08 \\
    10 & 0.000950425306504 & 3.71373256312e-11 \\
    11 & 6.728080753e-05 & 4.95408564432e-15 \\
    \hline
    12 & 1.87922375892e-06 & \\
    13 & 1.57356337999e-09 & \\
    14 & 2.89904818901e-15 & \\
    \hline
    \hline
    \gamma & 2.01181918489 \approx 2 & 1.66153582947 \approx \frac{1+\sqrt 5}2 \\
    \hline
  \end{array}
\end{align*}
Both methods attain their theoretical convergence rates in either iteration. However, we observe that the secant mode 
fares better with the problem with noise as it cuts through the latter, while the tangent Newton is thrown 
off for several iteration steps.

\section{Final Remarks}
The framework developed in the present paper, as well as in \cite{mother, son, proceed, SGE, ODE}, 
aims at presenting viable alternatives to current approaches to piecewise differentiability 
as it may occur, e.g., in nonsmooth nonlinear systems or ODEs with nonsmooth right hand side. 
The piecewise linearizations, which were first introduced in \cite{mother} can be obtained 
in an automated fashion by an adaptation of AD tools such as ADOL-C \cite{adolc}. 

The generalized Newton iterations introduced in Section 5 are intended as an alternative to 
semismooth Newton \cite{semismooth}. The quadratic convergence rate of the tangent version 
is in line with that of semismooth Newton. While for one step of semismooth Newton an appropriate element of the generalized derivative has to be calculated, the piecewise linear Newton's methods solve a piecewise 
linear system in each step. Both problems are NP-hard in general, but may be solvable in practice 
with essentially the same effort as an ordinary linear system. Hence, it is likely highly 
situation dependent, which approach yields better performance. 

It is our hope that, in combination with the formulas for numerically stable implementation 
of the secant linearization, the generalized Newton's methods can be developed into robust and 
stable workhorse algorithms, which might even outperform the quadratic methods on selected problems. 
For example on problems with oscillating noise we observed that the secant method required 
fewer Newton steps to converge, as it cuts through the oscillations. 


\section*{Acknowledgements}
The work for the article has been partially conducted within the Research Campus
MODAL funded by the German Federal Ministry of Education and Research
(BMBF) (fund number \(05\text{M}14\text{ZAM}\)).

\bibliographystyle{alphadin}
\bibliography{autodiff}

\end{document}